\documentclass[reqno,twoside,11pt]{amsart}

\usepackage{dsfont}

\usepackage{verbatim}

\theoremstyle{plain}

\numberwithin{equation}{section}

\usepackage{graphicx}
\usepackage{amssymb}
\usepackage{amsmath}
\usepackage{latexsym}
\usepackage{tikz}
\usepackage{mathrsfs}
\usepackage{textcomp}
\usepackage{fancyhdr}
\usepackage{pxfonts}
\usepackage{esint}
\usepackage{textcomp}
\usepackage{vmargin}

\newcommand{\supp}{\mathrm{supp}}

\newcommand{\sob}[1]{HW^{1,{#1}}(\Omega)}

\newcommand{\hor}{\nabla_{\mathbb{H}}}
\newcommand{\hess}{\nabla^2_{\mathbb{H}}}

\newcommand{\ai}{\left(\delta^2+|\hor{u}|^2\right)^{\frac{p-2}{2}}}

\newcommand{\heis}{\mathbb{H}}
\newcommand{\test}{C_0^{\infty}(\Omega)}

\newcommand{\norm}[1]{\left\lVert #1 \right\rVert}

\newcommand{\abs}[1]{\left\lvert #1 \right\rvert}
\newcommand{\A}{A_{k,r}^+}
\newcommand{\dx}{\;\mbox{d}x}
\newcommand{\osc}[1]{\mathrm{osc}_{#1}}

\theoremstyle{definition}

\newtheorem{Teo}{Theorem}[section]
\newtheorem{Def}[Teo]{Definition}
\newtheorem{Lemma}[Teo]{Lemma}
\newtheorem{rem}[Teo]{Remark}

\newtheorem{Prop}[Teo]{Proposition}

\newtheorem{Note}[Teo]{Note}

\newtheorem{theorem}{Theorem}[section]

\DeclareRobustCommand{\gobblefour}[4]{}

\begin{document}

\title[Regularity in the Heisenberg Group]
{On the $C^{1,\alpha}$ regularity of $p$-harmonic functions in the Heisenberg Group}
\author{Diego Ricciotti}
\address{Diego Ricciotti, University of Pittsburgh, Department of Mathematics, 
301 Thackeray Hall, Pittsburgh, PA 15260, USA}
\email{DIR17@pitt.edu}

\begin{abstract}
We present a proof of the local H\"older regularity of the horizontal derivatives of weak solutions to the $p$-Laplace equation in the Heisenberg group $\mathbb{H}^1$ for  $p>4 $. 
\end{abstract}
\maketitle
\tableofcontents

\let\thefootnote\relax
\let\thefootnote\relax

\section{Introduction}
We present a proof of the local H\"older regularity of derivatives of weak solutions to the $p$-Laplace equation in the Heisenberg group $\mathbb{H}^1$ for the range $p>4$. 
Our notation for the first Heisenberg group is $\heis=\mathbb{H}^1=(\mathbb{R}^{3},*)$. Here,  indicating points $x, y\in\heis$ by $x=(x_1,x_2,z)$ and $y=(y_1,y_2,s)$, the group operation is 
\begin{equation*}
x*y=(x_1,x_2,z)*(y_1,y_2,s)=\left(x_1+y_1, x_2+y_2, z+s+\frac{1}{2}(x_1y_2-x_2y_1)\right)
\end{equation*}
and a basis of left-invariant vector fields for the associated Lie algebra $\mathfrak{h}$ is given by 
\begin{equation*}
X_1=\partial_{x_1}-\frac{x_2}{2}\partial_z, \quad
X_2=\partial_{x_2}+\frac{x_1}{2}\partial_z \;\; \text{and}\;\;
T=\partial_z \;.
\end{equation*}
If $u:\Omega\longrightarrow \mathbb{R}$ is a function from an open subset of $\heis$ we indicate by $\hor{u}=(X_1u, X_2u)$ the horizontal gradient of $u$. 
We denote by $HW^{1,p}(\Omega)$ the Sobolev space of  functions $u$ such that both $u$ and $\hor{u}\in L^p(\Omega)$.

We study the regularity of solutions to the   $p$-Laplace equation:
\begin{equation}\label{plaplacedeg}
\sum_{i=1}^2X_i{\left(\abs{\hor{u}}^{p-2} X_iu\right)}=0 \quad \text{in} \,\; \Omega.
\end{equation}
The main result is the following:
\begin{theorem}\label{Teo oscillation}
Let $u\in HW^{1,p}(\Omega)$ be a weak solution of the $p$-Laplace equation \eqref{plaplacedeg} for $p>4$ and let  $B_{R_0}\subset\subset\Omega$. Then there exists $\beta=\beta(p)\in(0,1)$ such that for every $l\in\{ 1,2\}$ we have
\begin{equation*}
\osc{B_R}(X_lu) \leq C_p \norm{\hor{u}}_{L^\infty(B_{R_0})} \left( \frac{R}{R_0} \right)^\beta \quad \text{for all} \,\,R\leq \frac{R_0}{2},
\end{equation*}
where $C_p$ is a constant depending only on $p$.
\end{theorem}

In this work $B_r(x_0)$ denotes a Carnot-Carath\'{e}odory ball of radius $r$ and center $x_0$ (we omit the center when it is not essential). We recall that the Carnot-Carath\'{e}odory distance between $x$ and $y\in\heis$ is defined as
$$d_{cc}(x,y)
=\inf \{ l(\varGamma)\; |\; \varGamma\in S(x,y) \} \;.$$
Here $S(x,y)$ denotes the 
set of all horizontal subunitary curves joining $x$ and $y$, i.e.  absolutely continuous 
curves 
$\varGamma: [0,T]\longrightarrow \Omega$ such that 
$\varGamma'(t)=\sum_{j=1}^{2} \alpha_j(t)X_j(\varGamma(t))$
 for some 
real valued functions $\alpha_j$ with
$\sum_{j=1}^2\alpha_j(t)^2\leq 1$. 
The length of such a curve is defined to be $l(\varGamma)=T$. \\
Moreover we recall that the Lebesgue measure is the Haar measure of the group and the homogeneous dimension is $Q=4$.\\

To prove regularity results in general one considers a family of approximated non degenerate problems and tries to produce estimates independent of the non degeneracy parameter, in such a way that they can be applied to the degenerate equation by passing to the limit.
More precisely, here we consider the non degenerate equations
\begin{equation}\label{1plaplace}
\sum_{i=1}^2X_i{\left(\ai X_iu\right)}=0 \quad \text{in} \,\; \Omega
\end{equation}
for a parameter $\delta>0$ .
Equation \eqref{plaplacedeg} corresponds to the degenerate case $\delta=0$.
\\

The Heisenberg group  presents new challenges with respect to its Euclidean counterpart , since we only assume that $u$ is in the horizontal Sobolev space $HW^{1,p}$ and differentiating the equation produces terms involving the vertical derivative $Tu$, due to the non commutativity of the horizontal vector fields. This constitutes the main difficulty.\\

For $p=2$ it is now classical that the solutions of equation \eqref{plaplacedeg} are $C^\infty$ \cite{hormander}. 
For $p\neq 2$ the H\"older regularity of solutions of equations modeled on (\ref{1plaplace}) was established by Capogna and Garofalo \cite{capoholder} and Lu \cite{lu}.
Later Manfredi and Mingione \cite{manfredimingione} were able to prove $C^{1, \alpha}$ regularity in the non degenerate case for $2\leq p < c(n)<4$ and by adapting an argument used by Capogna they achieve $C^\infty$ regularity for this range of values of $p$.
The starting point is the integrability result for the vertical derivative $Tu\in L^p$ established by Domokos  for $1<p<4$ in \cite{domokos}, where he extends integrability results considered by Marchi for $1+\frac{1}{\sqrt{5}}<p<1+\sqrt{5}$ in \cite{marchiuno}, \cite{marchidue}.\\

Mingione, Zatorska-Goldstein and Zhong proved in \cite{minzo} that the Euclidean gradient of solutions to the non degenerate equation are $C^{1,\alpha}$ for $2\leq p<4$ and also that solutions to the degenerate equation are locally Lipschitz continuous for $2\leq p<4$.\\
 Zhong in  \cite{zhong} extended the Hilbert-Haar theory to the Heisenberg group setting and proved that solutions to the degenerate equation \eqref{plaplacedeg} are locally Lipschitz for the full range $1<p<\infty$.
For an account of this theory, further historical details and additional references see  \cite{me}. \\

As for the H\"older continuity of the horizontal derivatives for the degenerate equation \eqref{plaplacedeg} the only published result for $p\neq 2$ has been obtained by Manfredi and Domokos in \cite{cordes}, \cite{pnear} via the Cordes pertubation technique for $p$ near $2$.\\

The proof of the  H\"older continuity of the horizontal derivatives contained in this work uses the particular form of the equation in $\heis^1$ and new integration by parts for the second derivatives that produce weights of the form $(\delta^2+|\hor{u}|)^\frac{p-4}{2}$. This is the reason why our proof is only valid in the first Heisenberg group $\heis^1$ and for the range $p>4$.

\section{Preliminaries}

\subsection{The $p$-Laplace Equation}
We will consider the non degenerate $p$-Laplace equation \eqref{1plaplace}.
Denoting  $z=(z_1,z_2)\in\mathbb{R}^2$ and calling
\begin{equation*}
\begin{split}
a_i(z)=(\delta^2+|z|^2)^\frac{p-2}{2}z_i \quad \text{and} \quad
w=\delta^2+\abs{\hor{u}}^2 ,\\
\end{split}
\end{equation*}
 equation \eqref{1plaplace} rewrites as 
\begin{equation}\label{plaplacetype}
\sum_{i=1}^2X_i{a_i(\hor{u})}=0 \quad \text{in} \,\; \Omega
\end{equation}
and satisfies the following ellipticity and growth conditions for all $p>1$: 
\begin{equation}\label{estimates powers}
\begin{split}
\sum_{i,j=1}^2\partial_{z_j}a_i(\hor{u})\xi_i\xi_j &\geq c_p w^\frac{p-2}{2} |\xi|^2, \\
|a_i(\hor{u})|&\leq w^{\frac{p-1}{2}}, \\
|\partial_{z_j}a_i(\hor{u})|&\leq C_p w^\frac{p-2}{2} ,\\
|\partial_{z_s}\partial_{z_j}a_i(\hor{u})|&\leq C_p w^\frac{p-3}{2}
\end{split}
\end{equation}
and
\begin{equation}\label{estimates quotient}
\left\lvert\frac{\partial_{z_j}a_i(z)}{\partial_{z_l}a_l(z)}\right\rvert \leq C_p \quad \text{for all} \,\,i,j,l\in\{1,2\}.
\end{equation}
We remark that the proofs presented in this work  depend only on these properties, therefore they extend to more general equations of $p$-Laplacean type  as in \eqref{plaplacetype} for $a_i$ of class $C^2$ satisfying \eqref{estimates powers} and \eqref{estimates quotient}.

We say that a function $u\in\sob{p}$ is a weak solution of (\ref{1plaplace}) if 

\begin{equation}\label{2plaplaceweak}
\int_{\Omega}w^\frac{p-2}{2}\langle\hor{u},\hor{\varphi}\rangle \dx=0 \quad \text{for all} \;\varphi\in HW^{1,p}_0(\Omega) \;,
\end{equation}
where $HW^{1,p}_0(\Omega)$ is the closure of the space of $C^\infty$ compactly supported functions with respect to the Horizontal Sobolev norm.

\subsection{Previous Results}
We now collect some known results about the non degenerate equation \eqref{1plaplace} that will be used in the following sections. We refer to \cite{me} for a detailed presentation and complete proofs.

First we have that solutions to the non degenerate  $p$-Laplace equation \eqref{1plaplace} are smooth. This was  proved by Capogna in \cite{capognaregularity} for $p\geq 2$ and extended to the full range $1<p<\infty $ in \cite{me} Chapter $4$
 by adapting techniques of Domokos in \cite{domokos}. 

As a consequence we have  
\begin{equation}\label{mixed}
\sum_{i,j=1}^2\partial_{z_j}a_i(\hor{u}) X_iX_ju=0 \quad \text{a.e. in}\quad \Omega
\end{equation} hence we can express $X_1X_1u$ (respectively $X_2X_2u$)  in terms of $X_iX_ju$ where at least one index is a $2$ (respectively a $1$). This will be a crucial point later.

We now collect the equations satisfied by the horizontal and vertical derivatives ( see \cite{me}, Lemma 4.1):
\begin{Lemma}\label{Lemma:2eq}
The functions $X_1u$, $X_2u$ and $Tu$ are weak solutions respectively of the following equations (in $\Omega$):
\begin{align}
\sum_{i=1}^2 X_i\left( \sum_{j=1}^2 \partial_{z_j}a_i(\hor{u})X_jX_1u\right)
&+ \sum_{i=1}^2 X_i\left( \partial_{z_2}a_i(\hor{u})Tu \right)+T\left(a_2(\hor{u})\right)=0 \label{2eq1u} \\
\sum_{i=1}^2 X_i\left( \sum_{j=1}^2 \partial_{z_j}a_i(\hor{u})X_jX_2u\right)
&- \sum_{i=1}^2 X_i\left( \partial_{z_1}a_i(\hor{u})Tu \right)-T\left(a_1(\hor{u})\right)=0 \label{2eq2u}\\
\sum_{i=1}^2 X_i\left( \sum_{j=1}^2 \partial_{z_j}a_i(\hor{u})X_jTu\right)
&=0. \label{2eq3u}
\end{align}
\end{Lemma}

In \cite{zhong} Zhong established the weighted higher order integrability of $Tu$ as follows:
\begin{Lemma}\label{Lemma: estimate tu}
For all $q>4$ and $\xi\in C_0^\infty(\Omega)$ we have:
\begin{equation}
\int_{\Omega} \xi^{q}\, w^\frac{p-2}{2}\, \abs{Tu}^{q} \dx
\leq C(q) \left( \norm{\hor{\xi}}_{L^\infty}^2+\norm{\xi T\xi}_{L^\infty} \right)^\frac{q}{2} \int_{\supp{(\xi)}} w^\frac{p-2+q}{2} \dx,
\end{equation}
where $C(q)=C_p^\frac{q-2}{2}q^{q+8}$ and $C_p$ depends only on $p$.
\end{Lemma}
For the sake of completeness we give a proof in the Appendix.

\section{De Giorgi Classes in the Heisenberg Group}
We now describe a type of De Giorgi class in the Heisenberg group. These kind of spaces were introduced and studied by De Giorgi in the Euclidean case (see \cite{Giorgi}).
We will use the standard notation for super- (sub-) level sets of a measurable function
\begin{equation*}
\begin{split}
A_{k,r}^+&=A_{k,r}^+(f)=B_r\cap\{f>k \},\\
A_{k,r}^-&=A_{k,r}^-(f)=B_r\cap\{f<k \}.
\end{split}
\end{equation*}

\begin{Def}[De Giorgi class in the Heisenberg group]
Let $\Omega\subset\mathbb{H}$ be open, $\gamma$, $\chi$ positive real constants and $q>4$. A function $f\in HW^{1,2}_{\mbox{loc}}(\Omega)\cap L^\infty_{\mbox{loc}}(\Omega)$ belongs to the De Giorgi class $DG^+(\Omega, \gamma,\chi, q)$ if 
\begin{equation}\label{degiorgi}
\int_{A_{h,r'}^+}|\hor{f}|^2\dx 
\leq \frac{\gamma}{(r-r')^2}\sup_{B_r}{|(f-h)^+|^2}|A_{k,r}^+| +\chi |A_{k,r}^+|^{1-\frac{2}{q}}
\end{equation}
for some concentric balls $B_{r'}\subset B_r\subset\subset \Omega$ and levels $h\in\mathbb{R}$.
\end{Def}
In this section we consider an arbitrary ball $B_R\subset\subset\Omega$ and denote by $\displaystyle{M=M(R)=\sup_{B_R}f}$ and  $\displaystyle{m=m(R)=\inf_{B_R}{f}}$. \\

We are taking the following Lemma from \cite{kinnu}, Lemma 2.3, where it is proved in a more general setting.

\begin{Lemma}\label{levels}
Let $l>k$, $f\in HW^{1,1}_{loc}(\Omega)$, $B_r\subset\subset\Omega$. Then if $|B_r\setminus \A|>0$ we have:
\begin{equation}\label{levelz}
(l-k)|A_{l,r}^+|^{1-\frac{1}{4}} 
\leq \frac{Cr^4}{|B_r\setminus \A|} \int_{\A\setminus A_{l,r}^+} \abs{\hor{f}}\dx,
\end{equation}
where $C$ is a purely numeric constant.
\end{Lemma}

The next Lemma is adapted from Lemma 2.3 in \cite{Manfredi_thesis} and Lemma 6.1 in \cite{ladyzhen}.

\begin{Lemma}\label{2.3}
Let $0<\lambda_0, \lambda_1<1$ and $k<M$. Suppose $f\in DG^+(\Omega,\gamma,\chi, q)$ for \\
$h\in[k, \lambda_0k+(1-\lambda_0)M]$ and for $r'<r\in[\lambda_1R, R]$. Then there exists $\theta=\theta(\gamma, \lambda_0, \lambda_1)\in(0,1)$ such that if $$M-k\geq \chi^\frac{1}{2} R^{1-\frac{4}{q}},$$ then $$|A_{k,R}^+|\leq \theta |B_R|\quad \text{implies}\quad f\leq\lambda_0k+(1-\lambda_0)M \;\; \text{a.e. in}\;\; B_{\lambda_1R}.$$
\end{Lemma}

The following Lemma is adapted from Lemma 2.4 in \cite{Manfredi_thesis} and Lemma 6.2 in \cite{ladyzhen}.
\begin{Lemma}\label{Lemma 2.4}
Let $0<\lambda_1<1$ and $k<M$. Suppose $f\in DG^+(\Omega, \gamma,\chi,q)$ for $h\in[k,M]$ and for $r'=\lambda_1R$, $r=R$. If there exists a constant $0<C_0<1$ such that $|A_{k,\lambda_1R}^+|\leq C_0|B_{\lambda_1R}|$ then given $0<\theta<1$ there exists $s=s(\gamma, \lambda_1, C_0, \theta)\in\mathbb{N}$ such that  $$\text{if} \quad M-k\geq 2^s\chi^\frac{1}{2} R^{1-\frac{4}{q}} \quad \text{then} \quad |A_{k_s, \lambda_1R}^+|\leq \theta |B_{\lambda_1R}|,$$
where $k_s=k+( 1-2^{-s} )(M-k)$ is a level set between $k$ and $M$.
\end{Lemma}

Combining the previous Lemmas we get an estimate for the decay of the oscillation of functions in the De Giorgi class. We are adapting it from Lemma 2.5 in \cite{Manfredi_thesis} and from \cite{ladyzhen}.

\begin{Lemma}[Oscillation estimate]\label{Lemma osc}
Let $0<\lambda_1<1$ and suppose that for radii \\
 $r'<r\in[\lambda_1R,R]$ we have $f\in DG^+(\Omega, \gamma,\chi,q )$ for $h\in[\frac{m+M}{2}, M]$ and $-f\in DG^+(\Omega, \gamma,\chi,q )$ for $h\in[-M,-\frac{m+M}{2}]$ . Then there exists $A=A(\gamma, \lambda_1)\in(0,1)$ such that 
\begin{equation}\label{estima osc}
\osc{B_{\lambda_1R}}f\leq A \osc{B_{R}}f+BR^{1-\frac{4}{q}},
\end{equation}
where 
\begin{equation*}
B=\frac{\chi^\frac{1}{2}}{4(1-A)}.
\end{equation*}
\end{Lemma}

\section{Main Estimate}


From now on we will fix a ball $B_{R_0}\subset\subset\Omega$ and  for a concentric ball $B_R\subset B_{R_0}$ we introduce the notation 
\begin{equation*}
\begin{split}
\mu(R)=\max_{1\leq  l\leq 2}\norm{X_lu}_{L^\infty(B_R)} \quad \text{and} \quad
\lambda(R)=\frac{1}{2}\mu(R).
\end{split}
\end{equation*}

In this section we prove the following Proposition which contains the main estimates and constitutes the novel contribution of this work:
\begin{Prop}\label{Prop:fund}
Let $B_{R_0}\subset\subset\Omega$ and let $u\in HW^{1,p}(\Omega)$ be a weak solution of the non degenerate equation \eqref{1plaplace} for $p>4$. For every $0<r'<r<\frac{R_0}{2}$, $l=1,2$ and for every $q>\max\{4,2+\frac{4}{p-4}\}$ we have :
\begin{equation}\label{stima fondamentale}
\int_{B_{r'}} w^\frac{p-2}{2}\abs{\hor{(X_lu-k)^+}}^2 \dx\leq
\frac{C_p}{(r-r')^2}\int_{B_r} w^\frac{p-2}{2}|(X_lu-k)^+|^2 \dx
+\chi \abs{\A(X_lu)}^{1-\frac{2}{q}},
\end{equation}
\begin{equation}\label{stima fondamentale2}
\int_{B_{r'}} w^\frac{p-2}{2}\abs{\hor{(X_lu-k)^-}}^2 \dx\leq
\frac{C_p}{(r-r')^2}\int_{B_r} w^\frac{p-2}{2} |(X_lu-k)^-|^2 \dx
+ \chi \abs{A_{k,r}^-(X_lu)}^{1-\frac{2}{q}}.
\end{equation}
The inequalities \eqref{stima fondamentale} hold for levels $k\geq -\mu(R_0)$, while \eqref{stima fondamentale2} hold for levels $k\leq \mu(R_0)$. The constant $C_p$ depends only on $p$ and the parameter $\chi$ is given by
\begin{equation}\label{chi}
\chi=\frac{C_pq^6}{R_0^2} \left( \delta^2 +\mu(R_0)^2 \right)^\frac{p}{2} \abs{B_{R_0}}^\frac{2}{q}.
\end{equation} 
\end{Prop}

\begin{proof}
We will prove \eqref{stima fondamentale} for $l=1$, the other estimates follow in a similar fashion. We use the notation $v_l=(X_lu-k)^+=\max\{X_lu-k, 0\}$. Fix $0<r'<r<{\frac{R_0}{2}}$ and let $\phi=\xi^2v_1$, where $\xi$ is a cut-off function between $B_{r'}$ and $B_r$ with $\abs{\hor{\xi}}\leq \frac{C}{(r-r')}$. Denote $A_{k,r}^+(X_1u)$ for simplicity by $\A$ and we use the usual convention of sum on repeated indices.
Test equation \eqref{2eq1u} with $\phi$ to get:
\begin{equation*}
\begin{split}
J_1:=\int_{B_r} \xi^2 \, \partial_{z_j}a_i(\hor{u}) \,X_jX_1u \, X_iv_1 \dx
&= -2\int_{B_r}\xi \, \partial_{z_j}a_i(\hor{u})\, X_jX_1u  \,X_i\xi\, v_1 \dx \\
&\quad- \int_{B_r}\xi^2\,\partial_{z_2}a_i(\hor{u}) \,X_iv_1 \, Tu \dx \\
&\quad- 2\int_{B_r}\xi\,\partial_{z_2}a_i(\hor{u})\, X_i\xi \,Tu \, v_1\dx\\
&\quad- \int_{B_r} a_2(\hor{u})\,T(\xi^2 v_1) \dx \\
&\quad=:J_2+J_3+J_4+J_5.
\end{split}
\end{equation*}

Routine calculations using Young's inequality and \eqref{estimates powers} allow to estimate $J_i$ for $1\leq i\leq 4$ as follows
\begin{equation}\label{skipping}
\begin{split}
\int_{B_r} \xi^2\, w^\frac{p-2}{2} \, \abs{\hor{v_1}}^2 \dx
&\leq  C\int_{B_r} \abs{\hor{\xi}}^2 \, w^\frac{p-2}{2}\, v_1^2 \dx +  C\int_{\A} \xi^2 \, w^\frac{p-2}{2}\, \abs{Tu}^2 \dx \\
&+\left\lvert \int_{B_r} a_2(\hor{u})\,T(\xi^2 v_1) \dx \right\rvert.
\end{split}
\end{equation}
The new idea is to estimate  the last integral in the previous inequality  by integrating by parts twice. First, integrating by parts with respect to the field $T$ we get
\begin{equation*}
\begin{split}
\int_{B_r} a_2(\hor{u})\,T(\xi^2 v_1) \dx 
=-\int_{B_r} T(a_2(\hor{u}))\,\xi^2 v_1 \dx 
= -\int_{B_r} \partial_{z_j}a_2(\hor{u}) \, X_jTu \, \xi^2 \, v_1 \dx
\end{split}
\end{equation*}
and then integrating with respect to the fields $X_j$ for $j=1,2$ we obtain
\begin{equation*}
\begin{split}
 -\int_{B_r} \partial_{z_j}a_2(\hor{u}) \, X_jTu \, \xi^2 \, v_1 \dx
&= \int_{B_r} Tu \, X_j\left( \partial_{z_j}a_2(\hor{u}) \xi^2 v_1 \right) \dx\\
&= \int_{B_r} Tu\, \partial_{z_s}\partial_{z_j}a_2(\hor{u}) \, X_jX_su \, \xi^2 \, v_1 \dx\\
&\quad+ 2\int_{B_r} Tu\, \partial_{z_j}a_2(\hor{u}) \, \xi \, X_j\xi \, v_1 \dx\\
&\quad+ \int_{B_r} Tu\, \partial_{z_j}a_2(\hor{u}) \, \xi^2 \, X_jv_1 \dx \\
 &\quad=: J_{5,1}+J_{5,2}+J_{5,3}.
\end{split}
\end{equation*}
Note that $J_{5,2}$ and $J_{5,3}$ can be estimated respectively as $J_4$ and $J_3$.

Denoting $\displaystyle{J_{5,1}:=\sum_{s,j}J_{5,1}^{s,j}}$ we have:
\begin{equation}\label{J5j1}
\left\lvert\sum_{j}J_{5,1}^{1,j}\right\rvert\leq C_p\int_{B_r} \xi^2 \, w^\frac{p-3}{2} \,\abs{\hor{v_1}} \, v_1 \, \abs{Tu} \dx
\leq C_p\varepsilon \int_{B_r} \xi^2 \, w^\frac{p-2}{2} \, \abs{\hor{v_1}}^2 \dx 
+ \frac{C_p}{\varepsilon} \int_{B_r} \xi^2 \, w^\frac{p-4}{2} \, \abs{Tu}^2\, v_1^2 \dx 
\end{equation}
and
\begin{equation*}
\begin{split}
\left\vert J_{5,1}^{2,1}\right\rvert
\leq \int_{B_r}|\partial_{z_2}\partial_{z_1}a_2(\hor{u})\, X_1X_2u \, Tu\, \xi^2 \, v_1| \dx
&\leq  \int_{B_r}|\partial_{z_2}\partial_{z_1}a_2(\hor{u})\, X_2X_1u \, Tu|\, \xi^2 \, v_1 \dx \\
&+ \int_{B_r}|\partial_{z_2}\partial_{z_1}a_2(\hor{u})|\,  |Tu|^2\, \xi^2 \, v_1 \dx.
\end{split}
\end{equation*}
The first term of the last inequality can be estimated as in \eqref{J5j1}. For the other term we have:
\begin{equation*}
\begin{split}
\int_{B_r}|\partial_{z_2}\partial_{z_1}a_2(\hor{u})|\,  |Tu|^2\, \xi^2 \, v_1 \dx
\leq C_p\int_{B_r} \xi^2 \, w^\frac{p-3}{2} \, |Tu|^2 \, v_1 \dx 
&\leq C_p\int_{\A} \xi^2 \, w^\frac{p-2}{2}\, \abs{Tu}^2 \dx \\
&+ C_p\int_{B_r} \xi^2 \, w^\frac{p-4}{2} \, \abs{Tu}^2\, v_1^2 \dx.
\end{split}
\end{equation*}
Now another key step is to use the equation in \eqref{mixed} and \eqref{estimates quotient} to get:
\begin{equation*}
\begin{split}
\lvert J_{5,1}^{2,2}\rvert= \left\lvert\int_{B_r} \partial_{z_2}\partial_{z_2}a_2(\hor{u}) \, X_2X_2u \, \xi^2 \, v_1\, Tu \dx\right\rvert
&\leq C_p\int_{B_r}|\partial_{z_2}\partial_{z_2}a_2(\hor{u})\, X_1X_1u \, Tu|\, \xi^2 \, v_1 \dx \\
&+ C_p\int_{B_r}|\partial_{z_2}\partial_{z_2}a_2(\hor{u})\, X_2X_1u \, Tu|\, \xi^2 \, v_1 \dx \\
&+ C_p\int_{B_r}|\partial_{z_2}\partial_{z_2}a_2(\hor{u})\, X_1X_2u \, Tu|\, \xi^2 \, v_1 \dx \\
&=: F_1+F_2+F_3.
\end{split}
\end{equation*}
Note that $F_1$ and $F_2$ can be estimated as $J_{5,1}^{1,j}$ while $F_3$ can be estimated as $J_{5,1}^{2,1}$.

Choosing $\varepsilon$ small enough   \eqref{skipping} becomes:
\begin{equation}\label{base fund estimate}
\begin{split}
\int_{B_r} \xi^2\, w^\frac{p-2}{2} \, \abs{\hor{v_1}}^2 \dx
&\leq  C\int_{B_r} \abs{\hor{\xi}}^2 \, w^\frac{p-2}{2}\, v_1^2 \dx +  C\int_{\A} \xi^2 \, w^\frac{p-2}{2}\, \abs{Tu}^2 \dx + C\int_{B_r} \xi^2 \, w^\frac{p-4}{2} \, \abs{Tu}^2\, v_1^2 \dx \\
&\quad=: I_1+I_2+I_3.
\end{split}
\end{equation}
We only need to estimate $I_2$ and $I_3$. We use H\"older's inequality with exponent $q/2$ and Lemma \ref{Lemma: estimate tu}:
\begin{equation*}
\begin{split}
I_2 &\leq \left( \int_{\A} \xi^q\, w^\frac{p-2}{2}\, \abs{Tu}^q \dx\right)^\frac{2}{q} \left(  \int_{\A}w^\frac{p-2}{2} \dx\right)^{1-\frac{2}{q}} \\
&\leq \left( \int_{B_{R_0}} \eta^q\, w^\frac{p-2}{2}\, \abs{Tu}^q \dx\right)^\frac{2}{q} \left(  \int_{\A}w^\frac{p-2}{2} \dx\right)^{1-\frac{2}{q}}\\
&\leq \left( \left( \norm{\hor{\eta}}_{L^\infty}^2+\norm{\eta T\eta}_{L^\infty} \right)^\frac{q}{2} \int_{B_{R_0}} w^\frac{p-2+q}{2} \dx\right)^\frac{2}{q}
\left(\delta^2+\mu(r)^2\right)^{\frac{p-2}{2}(1-\frac{2}{q})} \abs{\A}^{1-\frac{2}{q}} \\
&\leq \frac{C_p\,q^6}{R_0^2} \left(\delta^2+\mu(R_0)^2\right)^\frac{p}{2} \abs{B_{R_0}}^\frac{2}{q}\abs{\A}^{1-\frac{2}{q}},
\end{split}
\end{equation*}
where $\eta$ is a cut-off function between $B_{\frac{R_0}{2}}$ and $B_{R_0}$ with $\abs{\hor{\eta}}\leq \frac{C}{R_0}$.
In a similar way  and noting that $v_1^2\leq 2(\delta^2+\mu(R_0)^2)$ for $k\geq -\mu(R_0)$ we get:
\begin{equation*}
\begin{split}
I_3 &\leq \left(\delta^2+\mu(R_0)^2\right) \left( \int_{\A} \xi^q\, w^\frac{p-2}{2}\, \abs{Tu}^q \dx\right)^\frac{2}{q}
\left(  \int_{\A}w^{\frac{p-4}{2}-\frac{2}{q-2}} \dx\right)^{1-\frac{2}{q}} \\
&\leq \left(\delta^2+\mu(R_0)^2\right) \left( \left( \norm{\hor{\eta}}_{L^\infty}^2+\norm{\eta T\eta}_{L^\infty} \right)^\frac{q}{2} \int_{B_{R_0}} w^\frac{p-2+q}{2} \dx\right)^\frac{2}{q}
\left(\delta^2+\mu(r)^2\right)^{\left(\frac{p-4}{2}-\frac{2}{q-2}\right)\left( 1-\frac{2}{q} \right)} \abs{\A}^{1-\frac{2}{q}} \\
&\leq \frac{C_p\,q^6}{R_0^2}\left( \delta^2+\mu(R_0)^2 \right)^\frac{p}{2} \abs{B_{R_0}}^\frac{2}{q} \abs{\A}^{1-\frac{2}{q}}.
\end{split}
\end{equation*}

\end{proof}

\begin{rem}
Note that the main difficulty in the proof is estimating the terms containing $Tu$. In particular in $J_5$ we integrate by parts twice to avoid dealing with terms involving $\hor{Tu}$. Then we use the equation in order to estimate terms with $X_2X_2u$ appropriately with quantities independent of $\delta$ or that can be absorbed in the right hand side.\\
\end{rem}

\section{Oscillation Estimate}
In this Section we prove our main result Theorem \ref{Teo oscillation}.
Recall that we fixed a ball $B_{R_0}\subset\subset\Omega$ and we now consider an arbitrary concentric ball $B_R\subset B_{\frac{R_0}{2}}$.

\begin{rem}\label{delta pos}
Let $u\in HW^{1,p}(\Omega)$ be a weak solution of the non degenerate equation \eqref{1plaplace} for $p>4$. 
For $\delta\geq\lambda(R)$  we easily get that for every $\lambda_1\in(0,1)$ there exists $A=A(p,\lambda_1)$ such that
\begin{equation*}
\osc{B_{\lambda_1R}}(X_lu) \leq A\osc{B_R}(X_lu)+BR^\alpha \quad \text{for every} \,\,  \,l\in\{ 1,2\},
\end{equation*}
where 
\begin{equation*}
\begin{split}
B=\frac{C_pq^\frac{6}{p}(\delta^2+\mu(R_0)^2)^\frac{1}{2}}{4(1-A)R_0^{\alpha}} \quad \text{and} \quad
\alpha=\left( 1-\frac{4}{q} \right)\frac{2}{p}.
\end{split}
\end{equation*}
\end{rem}
\begin{proof}
Since $\delta\geq\lambda(R)$  we can  get rid of the weight and obtain that $X_lu$ is in a De Giorgi class. Indeed 
from \eqref{stima fondamentale} we get
\begin{equation*}
\int_{B_{r'}} \abs{\hor{v_l}}^2 \dx\leq
\frac{C_p}{(r-r')^2}\int_{B_r} v_l^2 \dx
+\frac{2^{p-2}\chi}{\mu(R)^{p-2}} \abs{\A(X_lu)}^{1-\frac{2}{q}}
\end{equation*}
for all levels $k>-\mu(R_0)$ and radii $r'<r<R$.
Now if 
\begin{equation}\label{alt1a}
\mu(R)\geq \chi^\frac{1}{p}R^{\left(1-\frac{4}{q}\right)\frac{2}{p}}  
\end{equation}
 denoting by
\begin{equation*}
\chi '=C_p q^\frac{12}{p} (\delta^2+\mu(R_0)^2)\left(\frac{R}{R_0}\right)^{2\left( 1-\frac{4}{q} \right)\frac{2}{p}}\;R^{2\left( \frac{4}{q}-1\right)}
\end{equation*}
we get that $X_lu\in DG^+(B_{R_0},C_p, \chi ', q)$ for all levels $k>-\mu(R_0)$ and radii $r'<r<R$. \\
Analogously  from \eqref{stima fondamentale2} we get also $-X_lu\in DG^+(B_{R_0},C_p, \chi ', q)$ for all levels $k<\mu(R_0)$ and radii $r'<r<R$, hence we can apply the Oscillation estimate in Theorem \ref{Lemma osc} to get for any $\lambda_1\in(0,1)$ the existence of $A=A(p,\lambda_1)\in(0,1)$ such that for every  $l\in\{1,2\}$ we have
\begin{equation*}
\osc{B_{\lambda_1R}}(X_lu) \leq A\osc{B_R}(X_lu)+B'R^{1-\frac{4}{q}},
\end{equation*}
where $4(1-A)B'=(\chi ')^\frac{1}{2}$. By the definition of $\chi '$, and combining with the case when \eqref{alt1a} does not hold, we get the result.
\end{proof}

 We now consider the interesting case when the equation degenerates, namely $\delta<\lambda(R)$. Here we face an alternative: either the maximum $\mu(R)$ has the right 'H\"older decay' or the horizontal gradient $\hor{u}$ is bounded away from zero, and hence the equation behaves like the non degenerate case in Remark \ref{delta pos}. More precisely we have:
\begin{Prop}\label{Prop alternative}
Let $u\in HW^{1,p}(\Omega)$ be a weak solution of the non degenerate equation \eqref{1plaplace} for $p>4$ and 
consider $B_R\subset B_{\frac{R_0}{2}}$. Then there exist $\theta=\theta(p)\in(0,1)$ and $A=A(p)\in(0,1)$ such that: \\
{\bf Case 1.} If for some $l\in\{1,2\}$ we have 
either 
\begin{equation}\label{case1}
\abs{B_R\cap\left\{ X_l\geq\frac{1}{2}\mu(R) \right\}} \geq \theta |B_R|
\end{equation} 
or
\begin{equation}\label{case2}
\abs{B_R\cap\left\{ X_l\leq-\frac{1}{2}\mu(R) \right\}}  \geq \theta |B_R|
\end{equation} 
then 
\begin{equation*}
\text{either} \quad \mu(R)\leq c_p\chi^\frac{1}{p}R^{\left(1-\frac{4}{q}\right)\frac{2}{p}}
\quad \text{or} \quad
 |X_lu|\geq \frac{1}{32}\mu(R)\quad \text{a.e in}\,\,B_{R/2}
\end{equation*} 
where $c_p=2(4/3)^\frac{2}{p}$.\\
{\bf Case 2.} If for every $l\in\{1,2\}$ neither \eqref{case1} nor \eqref{case2} holds then 
\begin{equation}\label{alternative2estimate}
\mu(R/2) \leq A \mu(R)+B{R}^{\alpha},
\end{equation}
where 
\begin{equation*}
\begin{split}
B=\frac{C_p q^\frac{6}{p}}{2(1-A)}\frac{\mu(R_0)}{R_0^\alpha} \quad \text{and}\quad
\alpha=\left( 1-\frac{4}{q}\right)\frac{2}{p}.
\end{split}
\end{equation*}
\end{Prop}

\begin{proof}
{\bf Case 1.}\\
Consider \eqref{case2}. We will show that it implies $X_lu\leq-\frac{1}{32}\mu(R)$ provided $\mu(R)\geq c_p\chi^\frac{1}{p}R^{\left(1-\frac{4}{q}\right)\frac{2}{p}}$.
Define the auxiliary function $$V_l=|X_lu|^\frac{p}{2}\mbox{sign}(X_lu).$$ 
Observe that $|V_l|\leq (2\lambda(R))^\frac{p}{2}$ on $B_R$. Also
\begin{equation}\label{horV}
|\hor{V_l}|^2= \frac{p^2}{4}|X_lu|^{p-2}|\hor{X_lu}|^2.
\end{equation}
Denote by $h=|k|^\frac{p}{2}\mbox{sign}(k)=g(k)$ and
note that $\{X_l>k\}=\{V_l>h\}$ 
 therefore \eqref{stima fondamentale} becomes:
\begin{equation}\label{degiorgiperV1}
\begin{split}
\int_{A_{h,r'}^+(V_l)} |\hor{V_l}|^2\dx 
\leq C_p \int_{B_{r'}} w^\frac{p-2}{2} |\hor{v_l}|^2\dx
&\leq \frac{C_p}{(r-r')^2}(\mu(r)-k)^2(\delta^2+\mu(r)^2)^\frac{p-2}{2}|A_{h,r}^+(V_l)| \\
&\quad+ \chi |A_{h,r}^+(V_l)|^{1-\frac{2}{q}} \\
&\leq \frac{C_p}{(r-r')^2} (\lambda(R))^p |A_{h,r}^+(V_l)|+\chi |A_{h,r}^+(V_l)|^{1-\frac{2}{q}} 
\end{split}
\end{equation}
 for $k>-\lambda(R)$ and $r'<r\leq R$. 
Denoting by $H=H(R)=(\lambda(R))^\frac{p}{2}$ the inequality \eqref{degiorgiperV1} rewrites as:
\begin{equation}\label{degiorgiperV}
\int_{A_{h,r'}^+(V_l)} |\hor{V_l}|^2\dx
\leq \frac{C_p}{(r-r')^2} (H(R))^2 |A_{h,r}^+(V_l)|+\chi |A_{h,r}^+(V_l)|^{1-\frac{2}{q}} 
\end{equation}
for levels $h>-H(R)$ and radii $r'<r\leq R$.
Now denote $M(\frac{R}{2})=\sup_{B_{R/2}}V_l$.\\

{\bf Case a.}
$M(\frac{R}{2})<-\frac{H(R)}{4}$.\\
This means $X_lu<0$ in $B_{R/2}$ and after some algebraic manipulations 
\begin{equation*}
X_lu<
-\frac{\mu(R)}{32}\quad \text{on}\:\; B_{R/2}.
\end{equation*}

{\bf Case b.}
$M(\frac{R}{2})\geq-\frac{H(R)}{4}$.\\
For levels $h\in[-H(R),-H(R)/2]$ we have 
$\sup_{B_{R/2}}(V_l-h)^+
\geq 
\frac{H(R)}{4}$, therefore
\begin{equation*}
\frac{(H(R))^2}{16}
\leq \sup_{B_r}|(V_l-h)^+|^2
\quad \text{for}\;\; r\in[R/2,R].
\end{equation*}
Hence, from \eqref{degiorgiperV} we get that $V_l\in DG^+(B_{R_0},C_p, \chi, q)$ for levels $h\in[-H(R),-H(R)/2]$ and radii $r'<r\in[R/2,R]$. Apply Lemma \ref{2.3} with $k=-H(R)$, $\lambda_0=\frac{2^\frac{p}{2}+1/2}{2^\frac{p}{2}+1}$, $\lambda_1=1/2$ to get the existence of $\theta_1=\theta_1(p)\in(0,1)$ such that 
$|A_{-H,R}^+(V_l)|\leq\theta_1|B_R|$ implies $V_l\leq -\frac{H(R)}{2}$ on $B_{R/2}$, provided $M(R)+H(R)\geq \chi^\frac{1}{2} R^{1-\frac{4}{q}}$. 
This is true if
\begin{equation}\label{alterna}
\mu(R)\geq 2\left(\frac{4}{3}\right)^\frac{2}{p} \chi^\frac{1}{p}R^{\left( 1-\frac{4}{q} \right)\frac{2}{p}}.
\end{equation}
Then as in Case a, we obtain 
\begin{equation*}
X_lu<
-\frac{\mu(R)}{8}
\quad \text{in}\;\; B_{R/2}.
\end{equation*}
Observe that  $\{V_l>-H(R)\}=\{ X_lu>-\lambda(R) \}=\{ X_lu>-\mu(R)/2 \}$. Passing to the complements we have proved that there exists $\theta=1-\theta_1$ such that \eqref{case2} implies $X_lu\leq -\mu(R)/32$ on $B_{R/2}$, provided $\eqref{alterna}$.\\
If \eqref{case1} holds then we proceed similarly and we get the conclusion of Case 1.

{\bf Case 2.}\\
If for the $\theta$ found in Case 1 neither \eqref{case1} nor \eqref{case2} hold for any $l\in\{1,2\}$ then  
there exist $\frac{1}{2}<\lambda_1<1$ and $0<C_0<1$ such that
\begin{equation}\label{controalt1}
\abs{B_{\lambda_1R}\cap\left\{ X_l\geq\frac{1}{2}\mu(R) \right\}} \leq C_0 |B_{\lambda_1R}|
\end{equation}
and
\begin{equation}\label{controalt2}
\abs{B_{\lambda_1R}\cap\left\{ X_l\leq-\frac{1}{2}\mu(R) \right\}} \leq C_0 |B_{\lambda_1R}|
\end{equation}
are satisfied for every $l\in\{1,2\}$.
Considering levels $k\in[\frac{\mu(R)}{2}, \mu(R)]$, on $\{X_lu>k\}$ we have :
\begin{equation*}
k^{p-2}\leq w^\frac{p-2}{2}\leq C_p k^{p-2}.
\end{equation*}
Therefore in \eqref{stima fondamentale} we can get rid of the weight:
\begin{equation*}
\int_{B_{r'}} |\hor{v_l}|^2\dx 
\leq \frac{C}{(r-r')^2}\int_{B_r} v_l^2\dx +\frac{2^{p-2}\chi}{\mu(R)^{p-2}}\abs{\A(X_lu)}^{1-\frac{2}{q}}.
\end{equation*}
 and proceeding as in Remark \ref{delta pos}, if 
\begin{equation}\label{alt1}
\mu(R)\geq \chi^\frac{1}{p}R^{\left(1-\frac{4}{q}\right)\frac{2}{p}}  
\end{equation}
we get that $X_lu\in DG^+(B_{R_0},C_p, \chi ', q)$ for levels $k\in[\frac{\mu(R)}{2}, \mu(R)]$, radii $r'<r<R$
with $$\chi ' =C_p q^\frac{12}{p} \lambda(R_0)^2\left(\frac{R}{R_0}\right)^{2\left( 1-\frac{4}{q} \right)\frac{2}{p}}\;R^{2\left( \frac{4}{q}-1\right)}
.$$
Apply Lemma \ref{Lemma 2.4} with $\lambda_1$ and $C_0$ as in \eqref{controalt1}, $k=\frac{\mu(R)}{2}$ to conclude that given $\theta_0\in(0,1)$ there exists a natural number $s=s(p, \lambda_1,C_0,\theta_0)$ such that either
\begin{equation}\label{alt2}
\mu(R)\leq 
2^{s+1}(\chi ')^\frac{1}{2}R^{1-\frac{4}{q}}
\end{equation}
or
\begin{equation}\label{alt3}
|A_{k_s, \lambda_1R}^+|\leq \theta |B_{\lambda_1R}|
\end{equation}
where $k_s
=\mu(R)(1-2^{-s-1} )$. \\
Now in the case \eqref{alt3} we want to use Lemma \ref{2.3} for radii $r'<r\in[R/2,\lambda_1R]$, $k=k_s=(1-2^{-s-1})\mu(R)$, $\lambda_0=1/2$. This can be applied if
\begin{equation}\label{sottoalt4}
k_s<\sup_{B_{\lambda_1R}}(X_lu).
\end{equation}
Then we would conclude that either
\begin{equation}\label{alt4}
\begin{split}
X_lu\leq \frac{1}{2}k_s+\frac{1}{2}\mu(\lambda_1R)
&\leq \left(1-\frac{1}{2^{s+1}}\right)\frac{1}{2}\mu(R)+\frac{1}{2}\mu(R)\\
&=\mu(R)\left(1-\frac{1}{2^{s+2}}\right) \quad \text{a.e. in} \quad B_{\frac{R}{2}}
\end{split}
\end{equation}
or
\begin{equation}\label{alt5}
\sup_{\lambda_1R}(X_lu)\leq \left( 1-\frac{1}{2^{s+1}} \right)\mu(R)+(\chi ')^\frac{1}{2}R^{1-\frac{4}{q}}.
\end{equation}
If \eqref{sottoalt4} is not true then we get
\begin{equation}\label{alt6}
\sup_{B_{R/2}}(X_lu)\leq \sup_{\lambda_1R}(X_lu)\leq k_s=\left( 1-\frac{1}{2^{s+1}} \right)\mu(R).
\end{equation}
Repeating the same steps for $-X_lu$ using assumption \eqref{controalt2} and the estimate \eqref{stima fondamentale2}, we will find the same alternatives except instead of \eqref{alt4}-\eqref{alt6}  we will have
\begin{equation}\label{negativealt4}
\begin{split}
X_lu\geq -\mu(R)\left(1-\frac{1}{2^{s+2}}\right)-(\chi ')^\frac{1}{2}R^{1-\frac{4}{q}} \quad \text{a.e. in} \quad B_{\frac{R}{2}}.
\end{split}
\end{equation}
In conclusion, combining all the cases we get
$$
\mu(R/2)\leq \left( 1-\frac{1}{2^{s+2}} \right)\mu(R)+c_p q^\frac{6}{p} 2^{s+1}\lambda(R_0)\left(\frac{R}{R_0}\right)^{\left(1-\frac{4}{q}\right)\frac{2}{p}}.
$$
\end{proof}

Now we need the following technical Lemma, adapted from Lemma 7.3 in \cite{giustidirect}:

\begin{Lemma}\label{Lemma r}
Let $0<A,\lambda, \alpha<1$ with $A\neq\lambda^\alpha$ and $B, R_0\geq 0$. Let $\varphi:[0,+\infty[\longrightarrow[0,+\infty[$ be an increasing function such that
\begin{equation}\label{iter}
\varphi(\lambda R) \leq A \varphi(R)+BR^\alpha \quad \text{for all} \quad R\leq R_0.
\end{equation}
Then for every $R\leq R_0$ we have
\begin{equation}
\varphi(r)\leq \frac{1}{A}\left(\frac{r}{R}\right)^{\min\{\log_\lambda A, \alpha\}}\left[ \varphi(R)+\frac{BR^\alpha}{|A-\lambda^\alpha|} \right] 
\quad \text{for all} \quad r\leq R.
\end{equation}
\end{Lemma}
We finally prove Theorem \ref{Teo oscillation}.
\begin{proof}
[Proof of Theorem \ref{Teo oscillation}]
We prove the result for $u\in HW^{1,p}(\Omega)$ weak solution of the non degenerate equation \eqref{1plaplace}. Then we can obtain the estimate for solutions to the degenerate equation \eqref{plaplacedeg} by an approximation argument as in \cite{me} Theorem 5.3.\\
The alternatives in Proposition \ref{Prop alternative} can be combined in either
 \begin{equation}\label{s1}
\mu(R/2) \leq A \mu(R)+B{R}^{\alpha}
\end{equation}
or
\begin{equation}
 |X_lu|\geq \frac{1}{32}\mu(R)\quad \text{a.e in}\,\,B_{R/2}.
\end{equation} 
In this last case we have 
\begin{equation*}
w^\frac{p-2}{2}\geq \left( \frac{1}{32} \right)^{p-2}\mu(R)^{p-2} \quad \text{a.e. in}\quad B_{\frac{R}{2}}.
\end{equation*}
Since also 
\begin{equation*}
w^\frac{p-2}{2}\leq \left( \delta^2+\mu(R)^2 \right)^\frac{p-2}{2}\leq C_p\mu(R)^{p-2} \quad \text{in}\quad B_{R},
\end{equation*}
from the estimate \eqref{stima fondamentale} we get 
\begin{equation*}
\int_{B_{r'}} |\hor{v_l}|^2\dx \leq \frac{C}{(r-r')^2}\int_{B_r}v_l^2\dx + \frac{\chi}{\mu(R)^{p-2}}\abs{\A(X_lu)}^{1-\frac{2}{q}}
\end{equation*} 
for every $r'<r\leq R/2$ and for every level $k>-\mu(R_0)$. Now as before,
if 
\begin{equation}\label{alt again}
\mu(R)\geq \chi^\frac{1}{p}R^{\left( 1-\frac{4}{q} \right)\frac{2}{p}}
\end{equation}
 we get $v_l\in DG^+(B_{R_0}, C_p, \chi ', q)$ for every $r'<r\leq R/2$ and for every level $k>-\mu(R_0)$. The same is true for $-v_l$, with levels $k<\mu(R_0)$, so proceeding as in the proof of Remark \ref{delta pos}, we are in the position to apply the Oscillation Lemma \ref{Lemma osc} to conclude there exists $A=A(p)\in(0,1)$ such that
\begin{equation}\label{s2}
\osc{B_{R/4}}(X_lu) \leq A \osc{B_{R/2}}(X_lu)+BR^\alpha\leq A \osc{B_{R}}(X_lu)+BR^\alpha \quad \text{for all} \quad R\leq\frac{R_0}{2},
\end{equation}
where $B$ and $\alpha$ are as in Proposition \eqref{Prop alternative}.\\
Now apply Lemma \ref{Lemma r} to \eqref{s1} and \eqref{s2} with $\lambda=1/4$, $A$ and $B$ as given in \eqref{Prop alternative}. Noting that  $\osc{B_r}(X_lu)\leq 2\mu(r)$ we can combine all the estimates and hence the Theorem is proved with $\beta=\min\{-\log_4(A)\,,\,\alpha\}$.
\end{proof}

\begin{rem}
From the explicit expression of $\beta$ and $B$ we see that the estimate blows up when $q$ goes to infinity, hence the H\"older exponent found with this proof satisfies the constraint $0<\beta<\frac{2}{p}$.
\end{rem}

\section{Appendix: a proof of Lemma \ref{Lemma: estimate tu}}

We use the following estimates of Zhong \cite{zhong} ( see also \cite{me}, Lemmas 5.3 and 5.4).

\begin{Lemma}\label{2Lemma:3.5}
Let $q\geq 4$ and $\xi\in\test$. Then
\begin{equation}
\begin{split}
\int_{\Omega}\xi^{q}\,w^\frac{p-2}{2}\,|Tu|^{q-2}\,|\hess{u}|^2\dx  \leq
C_p^\frac{q-2}{2}(q-1)^{q-2}\norm{\hor{\xi}}_{L^\infty}^{q-2}\int_{\Omega}\xi^2\,w^\frac{p+q-4}{2}\,|\hess{u}|^2\dx \;.
\end{split}
\end{equation}
\end{Lemma}
\begin{Lemma}\label{Teo:3.1}
Let $q\geq 4$ and $\xi\in\test$. Then
\begin{equation*}
\int_{\Omega}\xi^2\,w^\frac{p+q-4}{2}\,|\hess{u}|^2\dx \leq 
C_p \,\left(\norm{\hor{\xi}}^2_{L^\infty}+\xi\norm{T\xi}_{L^\infty}\right)\, (q-1)^{10} \int_{\supp{(\xi)}} \,w^\frac{p+q-2}{2}\dx.
\end{equation*}                                                                                                                                               
\end{Lemma}
Lemma \ref{2Lemma:3.5} follows by using $\phi=\xi^q|Tu|^{q-2}X_iu$ as test functions in equations \eqref{2eq1u} and \eqref{2eq2u}, while Lemma \ref{Teo:3.1} follows  by using $\phi=\xi^2w^{\frac{q-2}{2}}X_iu$ and the estimate in Lemma \ref{2Lemma:3.5}.
\begin{proof}
[Proof of Lemma \ref{Lemma: estimate tu}]
Using $|Tu|\leq 2|\hess{u}|$ and Lemmas \ref{2Lemma:3.5} and \ref{Teo:3.1} we have for $q\geq 4$
\begin{equation}
\begin{split}
\int_{\Omega} \xi^{q}\, w^\frac{p-2}{2}\,|Tu|^q \dx
 &\leq 2\int_{\Omega} \xi^{q}\,w^\frac{p-2}{2}\,|Tu|^{q-2}\,|\hess{u}|^2 \dx\\
&\leq C_p^\frac{q-2}{2}(q-1)^{q-2}\norm{\hor{\xi}}_{L^\infty}^{q-2}\int_{\Omega} \xi^2\, w^\frac{p+q-4}{2}\,|\hess{u}|^2\dx  \\
&\leq C^\frac{q-2}{q}(q)^{q+8}\left(\norm{\hor{\xi}}^2_{L^\infty}+\xi\norm{T\xi}_{L^\infty}\right)^\frac{q}{2} \int_{\supp{(\xi)}} w^\frac{p+q-2}{2} \dx
\end{split}
\end{equation}
\end{proof}

\begin{Note}
While working on this paper,  the March 2016 preprint in the
arXiv \cite{capoledonne} was brought to my attention.
This manuscript contains a general statement that includes the regularity results
proved above. The proof in \cite{capoledonne} is based on  the proof put forward by Zhong in the
2009 preprint \cite{zhong},
 that I have credited throughout. The proof in this manuscript is different in the way it handles the term that is here dentoted by $J_5$ in the proof of Proposition \eqref{Prop:fund}. In \cite{capoledonne} additional iterations are used to control this term, while here we use a double integration by parts and the structure of the equation.
\end{Note}

\section*{Acknowledgements}
I thank Juan Manfredi for helpful comments and suggestions to improve this manuscript. 
\bibliography{reference_bib}\nocite{*}
\bibliographystyle{abbrv}

\end{document}